\documentclass[12pt,reqno]{amsart}

\usepackage{amsfonts, amsthm, amsmath}
\allowdisplaybreaks[4]

\usepackage{rotating}

\usepackage{tikz}

\usepackage{graphics}

\usepackage{amssymb}

\usepackage{amscd}

\usepackage[latin2]{inputenc}

\usepackage{t1enc}

\usepackage[mathscr]{eucal}

\usepackage{indentfirst}

\usepackage{graphicx}

\usepackage{graphics}

\usepackage{pict2e}

\usepackage{mathrsfs}

\usepackage{enumerate}
\usepackage[pagebackref]{hyperref}
\hypersetup{backref, pagebackref, colorlinks=true}
\usepackage{cite}
\usepackage{color}
\usepackage{epic}
\usepackage{hyperref} 
\numberwithin{equation}{section}
\theoremstyle{plain}

\newtheorem{theorem}{Theorem}[section]

\newtheorem{lemma}[theorem]{Lemma}

\newtheorem{corollary}[theorem]{Corollary}

\theoremstyle{definition}

\newtheorem{remark}[theorem]{Remark}

\newtheorem{?}[theorem]{Problem}

\def\boxit#1{\leavevmode\hbox{\vrule\vtop{\vbox{\kern.33333pt\hrule
    \kern1pt\hbox{\kern1pt\vbox{#1}\kern1pt}}\kern1pt\hrule}\vrule}}

\usepackage{collectbox}



\newcommand{\f}[1]{\ifthenelse{\equal{#1}{1}}{(q;q)_\infty}{(q^{#1};q^{#1})_{\infty}}}
\newcommand{\ee}[1]{e\left(#1\right)}

\begin{document}
\title[Some inequalities for Garvan's bicrank function]{Some inequalities for Garvan's bicrank function of 2-colored partitions}

\author[S. Chern]{Shane Chern}
\address[Shane Chern]{Department of Mathematics, The Pennsylvania State University, University Park, PA 16802, USA}
\email{shanechern@psu.edu}

\author[D. Tang]{Dazhao Tang}

\address[Dazhao Tang]{College of Mathematics and Statistics, Chongqing University, Huxi Campus LD206, Chongqing 401331, P.R. China}
\email{dazhaotang@sina.com; dazhaotang@cqu.edu.cn}

\author[L. Wang]{Liuquan Wang}

\address[Liuquan Wang]{School of Mathematics and Statistics, Wuhan University, Wuhan 430072, P.R. China}
\email{mathlqwang@163.com; wanglq@whu.edu.cn}

\date{\today}

\begin{abstract}
In order to provide a unified combinatorial interpretation of congruences modulo $5$ for 2-colored partition functions, Garvan introduced a bicrank statistic in terms of weighted vector partitions. In this paper, we obtain some inequalities between the bicrank counts $M^{*}(r,m,n)$ for $m=2$, $3$ and $4$ via their asymptotic formulas and some $q$-series techniques. These inequalities are parallel to Andrews and Lewis' results on the rank and crank counts for ordinary partitions.
\end{abstract}



%

\maketitle

\section{Introduction}

A \emph{partition} \cite{Andr1976} of a nonnegative integer $n$ is a finite weakly decreasing sequence of positive integers whose sum equals $n$. One of Ramanujan's most beautiful works \cite{Ram1927} deals with congruences modulo $5$, $7$ and $11$ for the partition function $p(n)$, which counts the number of partitions of $n$. In 1944, Dyson \cite{Dys1944} defined the \emph{rank} of a partition and conjectured that this partition statistic can combinatorially interpret Ramanujan's congruences modulo $5$ and $7$. This was later confirmed by Atkin and Swinnerton-Dyer \cite{ASD1954}. Dyson also conjectured a partition statistic, which he called \emph{crank}, to provide a unified combinatorial interpretation of Ramanujan's three congruences. This partition statistic was found after over four decades by Andrews and Garvan \cite{AG1988}. In recent years, the two partition statistics motivate a tremendous amount of research.

A partition is called a \emph{$2$-colored partition} if each part is receiving a color from the set of two prescribed colors. Let $p_{-2}(n)$ count the number of 2-colored partitions of $n$. Then
\begin{align*}
\sum_{n=0}^{\infty}p_{-2}(n)q^{n} &=\dfrac{1}{(q;q)_{\infty}^{2}}.
\end{align*}
Here and in what follows, we assume $|q|<1$ and adopt the following customary notation on partitions and $q$-series:
\begin{align*}
(a;q)_{\infty} &:=\prod_{n=0}^{\infty}(1-aq^{n}),\\
(a_{1},a_{2},\cdots,a_{n};q)_{\infty}& :=(a_{1};q)_{\infty}(a_{2};q)_{\infty}\cdots(a_{n};q)_{\infty}.
\end{align*}

The following Ramanujan-type congruences for $p_{-2}(n)$ were proved by Hammond and Lewis \cite{HL2004} as well as Baruah and Sarmah \cite[Eq.~(5.4)]{BS2013}:
\begin{align}\label{partition pairs mod 5}
p_{-2}(5n+2)\equiv p_{-2}(5n+3)\equiv p_{-2}(5n+4)\equiv0\pmod{5}.
\end{align}
Moreover, Hammond and Lewis \cite{HL2004} defined a partition statistic which they named ``birank'', and showed that this birank may give a combinatorial proof of the above three congruences.

Analogous to the Andrews--Garvan--Dyson crank function for partitions, Andrews \cite{And2008} introduced his version of bicrank function for $2$-colored partitions, which can combinatorially interpret
$$p_{-2}(5n+3)\equiv 0\pmod{5}.$$
In 2010, Garvan \cite{Gar2010} further modified Andrews' bicrank function so that it can give a unified combinatorial proof of all the three congruences in \eqref{partition pairs mod 5}. We refer to \cite[Eq.~(6.12)]{Gar2010} for the definition of Garvan's bicrank function.

Due to the universality of Garvan's bicrank function, throughout this paper, what we mean by bicrank will be Garvan's version. Let $M^{*}(m,n)$ count the number of $2$-colored partition of $n$ with bicrank $m$. Garvan showed that $M^{*}(m,n)$ satisfies the following generating function (cf.~\cite[Eq.~(6.17)]{Gar2010})
\begin{align}
\sum_{m=-\infty}^{\infty}\sum_{n=0}^{\infty}M^{*}(m,n)z^{m}q^{n}=\dfrac{(q;q)_{\infty}^{2}}{(zq,z^{-1}q,z^{2}q,z^{-2}q;q)_{\infty}},\label{bicrank:gf}
\end{align}
from which he proved that, for any integer $n\geq0$,
\begin{align*}
M^{*}(0,5,5n+2) &=M^{*}(1,5,5n+2)=\cdots=M^{*}(4,5,5n+2)=\dfrac{p_{-2}(5n+2)}{5},\\
M^{*}(0,5,5n+4) &=M^{*}(1,5,5n+4)=\cdots=M^{*}(4,5,5n+4)=\dfrac{p_{-2}(5n+4)}{5},\\
M^{*}(0,5,5n+3) &\equiv M^{*}(1,5,5n+3)\equiv\cdots\equiv M^{*}(4,5,5n+3)\pmod{5},
\end{align*}
where $M^{*}(j,k,n):=\sum_{m\equiv j\pmod{k}}M^{*}(m,n)$ is the number of 2-colored partitions of $n$ with bicrank congruent to $j$ modulo $k$.

On the other hand, in 2000, Andrews and Lewis \cite{AL2000} studied the sign patterns for the rank and crank counts of ordinary partitions. The following is an example of their results for the crank function.
\begin{theorem}[Andrews--Lewis]
For all $n\geq0$,
\begin{align*}
M(0,2,2n) &>M(1,2,2n),\\
M(0,2,2n+1) &<M(0,2,2n+1),
\end{align*}
where $M(r,m,n)$ counts the number of partitions of $n$ with crank congruent to $r$ modulo $m$.
\end{theorem}

In the same paper, Andrews and Lewis also established some inequalities between the crank counts modulo $4$ for ordinary partitions and conjectured several sign patterns for the rank counts modulo 3 as well as the crank counts modulo 3 and 4. Their conjectures were later confirmed by Kane \cite{Kan2004} and Chan \cite{chan2005}. Subsequently, Chan and Mao \cite{CM2016} obtained a refined result on the rank and crank counts modulo 3 through standard $q$-series techniques.

These works motivate us to investigate Garvan's bicrank counts modulo $2$, $3$ and $4$ and obtain similar inequalities (see Theorems \ref{THM:mod 2}, \ref{THM:mod 3} and \ref{THM:mod 4} below).

The remainder of this paper is organized as follows. In Section \ref{sec:main results}, we present our main results. In Section~\ref{sec:circle}, to prove the inequalities satisfied by Garvan's bicrank counts, we will apply the circle method, and hence a detailed background description will be given. We also provide some slight modifications of Chan's work in \cite{chan2005}. In the next two sections, we will show two asymptotic formulas and deduce the desired inequalities. Then in Section~\ref{sec:ele}, we provide some elementary proofs for several special cases of the inequalities. We conclude with some remarks in the last section.

\section{Main results}\label{sec:main results}
Along the direction of Andrews and Lewis \cite{AL2000}, we prove the following inequalities together with other results for bicrank counts.

\begin{theorem}\label{THM:mod 2}
For all $n\geq0$,
\begin{align*}
M^{*}(0,2,n)& >M^{*}(1,2,n)\quad\quad \emph{if}~n\equiv 0 \pmod{2},\\
M^{*}(0,2,n)& <M^{*}(1,2,n)\quad\quad \emph{if}~n\equiv 1 \pmod{2}.
\end{align*}
\end{theorem}

This result is almost trivial. Taking $z=-1$ in \eqref{bicrank:gf} yields
\begin{align}\label{gf:mod 2}
\sum_{n=0}^{\infty}\left(M^{*}(0,2,n)-M^{*}(1,2,n)\right)q^{n}&=\dfrac{(q;q)_{\infty}^{2}}{(-q;q)_{\infty}^{2}(q;q)_{\infty}^{2}}=
\dfrac{(q;q)_{\infty}^{2}}{(q^{2};q^{2})_{\infty}^{2}}=(q;q^{2})_{\infty}^{2}.
\end{align}
We further take $q\to -q$, then the right-hand side becomes the generating function of pairs of partitions into distinct odd parts, and hence has its Fourier expansion coefficients being positive. This immediately leads to the desired result.

\begin{remark}
Recently, the second author and Fu \cite[Theorem 1.4]{FT2018} obtained two similar inequalities for the birank. Interestingly, assuming that $M_{2}(r,m,n)$ counts the number of $2$-colored partitions of $n$ with birank congruent to $r$ modulo $m$, then (cf.~\cite[Remark 1.2 and Eq.~(2.1)]{FT2018})
\begin{align}
M^{*}(0,2,n)-M^{*}(1,2,n)=M_{2}(0,2,n)-M_{2}(1,2,n).\label{birank-bicrank}
\end{align}
It would be interesting if one can find a combinatorial interpretation of \eqref{birank-bicrank}.
\end{remark}

\begin{remark}
In 2015, Kot{\v e}{\v s}ovec \cite[p.~13]{Kot2015} obtained the asymptotic formula for the coefficients of
$$\frac{1}{(-q;q)_\infty^m},$$
where $m$ is a positive integer. The $m=2$ case yields
\begin{equation*}
M^{*}(0,2,n)-M^{*}(1,2,n)\sim \frac{(-1)^n}{2^{3/2} 3^{1/4} n^{3/4}}\exp\left(\pi\sqrt{\frac{n}{3}}\right).
\end{equation*}
The interested readers may compare this formula with our Theorems \ref{th:main-asy} and \ref{th:main-asy-b}.
\end{remark}

The cases when $m=3$ and $4$ are more intriguing. We first have
\begin{theorem}\label{THM:mod 3}
For all $n\geq0$,
\begin{align}
M^{*}(0,3,n)&>M^{*}(1,3,n)\quad\quad  \emph{if}~n\equiv 0,2 \pmod{3},\label{mod 3:pos}\\
M^{*}(0,3,n)&<M^{*}(1,3,n)\quad\quad \emph{if}~n\equiv 1 \pmod{3},\label{mod 3:neg}
\end{align}
except for $n=5$.
\end{theorem}

This result is a consequence of an asymptotic formula for $M^{*}(0,3,n)-M^{*}(1,3,n)$. Before stating this formula, we recall that the modified Bessel functions of the first kind is defined as
\begin{align}
I_{s}(x)=\sum_{m=0}^{\infty}\frac{1}{m!\Gamma(m+s+1)}\left(\frac{x}{2}\right)^{2m+s},
\end{align}
where $\Gamma(z)$ is the gamma function. It is known \cite[p.~377, (9.7.1)]{handbook} that, for fixed $s$, when $|\arg x|<\frac{\pi}{2}$,
\begin{align}\label{Bessel-order}
I_{s}(x)\thicksim \frac{e^x}{\sqrt{2\pi x}}\left(1-\frac{4s^2-1}{8x}+\frac{(4s^2-1)(4s^2-9)}{2!(8x)^2}-\cdots \right).
\end{align}
\begin{theorem}\label{th:main-asy}
For $n\ge 1$,
\begin{equation}\label{eq:main-asy-a-1}
M^{*}(0,3,n)-M^{*}(1,3,n)=c(n) \; I_0\left(\frac{2\pi \sqrt{n-\frac{1}{12}}}{3\sqrt{3}}\right)+E(n),
\end{equation}
where
\begin{align*}
c(n)=\begin{cases}
\frac{4\pi}{3} \cos\frac{2\pi}{9} & \emph{if}~n\equiv 0 \pmod{3},\\
-\frac{4\pi}{3} \cos\frac{\pi}{9} & \emph{if}~n\equiv 1 \pmod{3},\\
\frac{4\pi}{3} \sin\frac{\pi}{18} & \emph{if}~n\equiv 2 \pmod{3},
\end{cases}
\end{align*}
and
\begin{align*}
|E(n)|&\le 173.1\; \sqrt{n-\frac{1}{12}}\; \left(\log \sqrt{2\pi \left(n-\frac{1}{12}\right)}+1\right) + 74.3\; \sqrt{n-\frac{1}{12}}\nonumber\\
&\quad+ 2.8\; \sqrt{n-\frac{1}{12}} \; \exp\left({\dfrac{\pi \sqrt{n-\frac{1}{12}}}{3\sqrt{3}}}\right).
\end{align*}
\end{theorem}

We also have
\begin{theorem}\label{THM:mod 4}
For all $n\geq0$,
\begin{align}
M^{*}(0,4,n)&>M^{*}(2,4,n)\quad\quad  \emph{if}~n\equiv 0,2,3,7 \pmod{8},\label{mod 4:pos}\\
M^{*}(0,4,n)&<M^{*}(2,4,n)\quad\quad \emph{if}~n\equiv 1,4,5,6 \pmod{8},\label{mod 4:neg}
\end{align}
except for $n=4$ and $20$.
\end{theorem}

Again, this is a consequence of the following asymptotic formula.

\begin{theorem}\label{th:main-asy-b}
For $n\ge 1$,
\begin{align}
&M^{*}(0,4,n)-M^{*}(2,4,n)\nonumber\\
&\qquad=c_{1}(n) \; I_0\left(\frac{\pi \sqrt{n-\frac{1}{12}}}{2\sqrt{3}}\right)+c_{2}(n) \; I_0\left(\frac{\pi \sqrt{n-\frac{1}{12}}}{4\sqrt{3}}\right)+E(n),\label{eq:main-asy-b-1}
\end{align}
where
\begin{align*}
c_{1}(n)=\begin{cases}
-\pi & \emph{if}~n\equiv 1 \pmod{4},\\
\pi & \emph{if}~n\equiv 3 \pmod{4},\\
0 & \emph{otherwise},
\end{cases}
\end{align*}
\begin{align*}
c_{2}(n)=\begin{cases}
\pi \sin\frac{\pi}{8} & \emph{if}~n\equiv 0 \pmod{8},\\
\pi \cos\frac{\pi}{8} & \emph{if}~n\equiv 2 \pmod{8},\\
-\pi \sin\frac{\pi}{8} & \emph{if}~n\equiv 4 \pmod{8},\\
-\pi \cos\frac{\pi}{8} & \emph{if}~n\equiv 6 \pmod{8},\\
0 & \emph{otherwise},
\end{cases}
\end{align*}
and
\begin{align*}
|E(n)|&\le 224.2\; \sqrt{n-\frac{1}{12}}\; \left(\log \sqrt{2\pi \left(n-\frac{1}{12}\right)}+1\right) + 55.6\; \sqrt{n-\frac{1}{12}}\nonumber\\
&\quad+ 2.4\; \sqrt{n-\frac{1}{12}} \; \exp\left({\dfrac{\pi \sqrt{n-\frac{1}{12}}}{6\sqrt{3}}}\right).
\end{align*}
\end{theorem}

\section{The circle method}\label{sec:circle}

One powerful tool to obtain the asymptotics in Theorems \ref{th:main-asy} and \ref{th:main-asy-b} is the circle method. Here we shall use Chan's version in \cite{chan2005}, whose approach follows from Rademacher (cf.~\cite[Chap.~5]{Andr1976}) and a later modification due to Kane \cite{Kan2004}.

For convenience, we adopt Chan's notation in \cite{chan2005}. We also use the conventional notation $\ee{x}:=e^{2\pi i x}$ for complex $x$.

\subsection{Cauchy's integral formula and Farey arcs}

Let $f(x)=\sum_{n=0}^\infty a(n) x^n$ be a holomorphic function on a simply connected domain containing the origin. Cauchy's integral formula tells us
\begin{equation*}
a(n)=\frac{1}{2 \pi i} \oint_{|x|=r} \frac{f(x)}{x^{n+1}}\ dx,
\end{equation*}
where the contour integral is taken counter-clockwise. We set $r=e^{-2\pi /N^2}=e^{-2 \pi \varrho}$ with $N$ to be determined.

We next dissect the circle by Farey arcs. It follows that
\begin{equation*}
a(n)=\sum_{1\le k\le N} \sum_{\substack{0\le h\le k\\ (h,k)=1}} \ee{-\frac{nh}{k}} \int_{\xi_{h,k}} f\big(\ee{h/k+i \varrho +\phi}\big)\ee{-n \phi} e^{2 \pi n \varrho}\ d\phi,
\end{equation*}
where $h/k$ is a Farey fraction of order $N$ and $\xi_{h,k}$ denotes the interval $[-\theta'_{h,k},\theta''_{h,k}]$ with $-\theta'_{h,k}$ and $\theta''_{h,k}$ being the positive distances from $h/k$ to its neighboring mediants.

At last, making the following changes of variables $z=k(\varrho -i \phi)$ and $\tau = (h+i z)/k$ yields
\begin{equation}\label{eq:cauchy-var}
a(n)=\sum_{1\le k\le N} \sum_{\substack{0\le h\le k\\ (h,k)=1}} \ee{-\frac{nh}{k}} \int_{\xi_{h,k}} f\big(\ee{\tau}\big)\ee{-n \phi} e^{2 \pi n \varrho}\ d\phi.
\end{equation}

\subsection{Functional equations}

Let $q:=\ee{\tau}$ with $\tau\in\mathbb{H}$, the upper half complex plane. For any $\gamma\in SL_2(\mathbb{Z})$, we define the fractional linear transformation
$$\gamma(\tau):=\frac{a\tau +b}{c\tau + d}, \quad \text{with }\gamma=\begin{pmatrix}a & b\\c& d\end{pmatrix}.$$
We also write
$$F(q)=F\big(\ee{\tau}\big):=\frac{1}{\f{1}}.$$

The functional equation for Dedekind eta-function (cf.~\cite[pp.~52--61]{Apo1990}) tells us that when $c>0$
\begin{equation}\label{eq:F-gamma}
F\big(\ee{\tau}\big)=\ee{\frac{\tau-\gamma(\tau)}{24}-\frac{s(d,c)}{2}+\frac{a+d}{24c}} \sqrt{-i (c\tau +d)}\ F\big(\ee{\gamma(\tau)}\big),
\end{equation}
where $s(d,c)$ is the Dedekind sum defined by
\begin{equation*}
s(d,c):=\sum_{n \bmod{c}} \bigg(\bigg(\frac{dn}{c}\bigg)\bigg)\bigg(\bigg(\frac{n}{c}\bigg)\bigg)
\end{equation*}
with
$$((x)):=\begin{cases}
x-\lfloor x\rfloor -1/2 & \text{if $x\not\in \mathbb{Z}$},\\
0 & \text{if $x\in \mathbb{Z}$},
\end{cases}$$
and the square root is taken on the principal branch, with $z^{1/2}>0$ for $z>0$.

Recall that $\gcd(h,k)=1$. For $m$ such that $\gcd(m,k)=1$, we put $h_m=mh$ and choose an integer $h'_m$ such that $h'_m h_m \equiv -1 \pmod{k}$ with $b_m=(h'_m h_m+1)/k$. Let
$$\gamma_{(m,k)}=\begin{pmatrix}h'_m & -b_m\\k & -h_m\end{pmatrix}.$$
It is easy to check that $\gamma_{(m,k)}\in SL_2(\mathbb{Z})$.

Chan \cite{chan2005} obtained the functional equation between $F(e(\tau))$ and $F(e(\gamma_{(m,k)}(\tau)))$ provided $\gcd(m,k)=1$. However, it will make our subsequent arguements more transparent if we slightly modify Chan's transformation formula for general $m$.

\begin{lemma}[Modification of {\cite[(3.13)]{chan2005}}]
Suppose that $\gcd(m,k)=d$ with $m=d m'$ and $k= d k'$. We have
\begin{equation}\label{eq:gamma}
\gamma_{(m',k')}(m\tau)=\frac{h'_{m'}}{k'}+\frac{1}{m'k'z}i
\end{equation}
and
\begin{equation}\label{eq:F-m-gamma}
F\big(\ee{m\tau}\big)=e^{\frac{\pi}{12 k'}\left(\frac{1}{m'z}-m'z\right)} \ee{\frac{s(m'h,k')}{2}} \sqrt{m' z}\ F\big(\ee{\gamma_{(m',k')}(m\tau)}\big).
\end{equation}
\end{lemma}

\begin{proof}
Recall that
$$\tau=\frac{h+iz}{k}=\frac{h+iz}{d k'}.$$
Hence noticing that $m/d=m'$ and $m'h=h_{m'}$, we have
\begin{align*}
\gamma_{(m',k')}(m\tau)&=\dfrac{h'_{m'}\cdot m\frac{h+iz}{d k'}-b_{m'}}{k' \cdot m\frac{h+iz}{d k'}-h_{m'}}=\frac{h'_{m'} h_{m'}+h'_{m'}(im'z)-(h'_{m'}h_{m'}+1)}{h_{m'} k'+k'(im'z)-h_{m'}k'}\\
&=\frac{h'_{m'}}{k'}+\frac{1}{m'k'z}i.
\end{align*}
Furthermore, we obtain \eqref{eq:F-m-gamma} by taking $\tau\mapsto m\tau$ and $\gamma\mapsto \gamma_{(m',k')}$ in \eqref{eq:F-gamma}.
\end{proof}

\subsection{Necessary bounds}

Chan \cite{chan2005} also obtained most necessary bounds.

Let $|\xi_{h,k}|$ denote the length of the interval $\xi_{h,k}$. It follows from \cite[Eq.~(3.14)]{chan2005} that
\begin{equation}\label{eq:xi-bound}
\frac{1}{kN}\le |\xi_{h,k}| \le \frac{2}{kN}.
\end{equation}

Noting that $z=k(\varrho -i \phi)$, hence
\begin{equation}\label{eq:Re-z-bound}
\Re(z)=\frac{k}{N^2}.
\end{equation}
This implies that
\begin{equation}\label{eq:z-bound}
|z|\ge \frac{k}{N^2}.
\end{equation}
Chan also proved that (cf.~\cite[Eq.~(3.16)]{chan2005})
\begin{align}\label{eq:Re-1-z-bound}
\Re\left(\frac{1}{z}\right)\ge \frac{k}{2}.
\end{align}

We show explicitly

\begin{lemma}
Suppose that $\gcd(m,k)=d$ with $m=d m'$ and $k= d k'$. We have
\begin{align}
\left|F\big(\ee{\gamma_{(m',k')}(m\tau)}\big)\right|&\le \exp\left(\frac{e^{-\pi \gcd(m,k)^2/m}}{\left(1-e^{-\pi \gcd(m,k)^2/m}\right)^2}\right),\label{eq:bound-1}\\
\left|\frac{1}{F\big(\ee{\gamma_{(m',k')}(m\tau)}\big)}\right|&\le \exp\left(\frac{e^{-\pi \gcd(m,k)^2/m}}{\left(1-e^{-\pi \gcd(m,k)^2/m}\right)^2}\right).\label{eq:bound-2}
\end{align}
\end{lemma}

\begin{proof}
We write in this proof $q=\ee{y}$ with $y\in \mathbb{H}$. Recall that $p(n)$ denotes the number of partitions of $n$ with the  convention that $p(0)=1$. It is clear that $p(n)>0$ for all $n\ge 0$.

One the one hand, we have
$$|F(q)|\le \sum_{n=0}^\infty p(n) |q|^n = F(|q|).$$
On the other hand,
\begin{equation*}
\left|\frac{1}{F(q)}\right|=\prod_{k=0}^\infty \left|1-q^k\right|\le \prod_{k=0}^\infty \left(1+|q|^k\right)\le \prod_{k=0}^\infty \frac{1}{1-|q|^k}=F(|q|).
\end{equation*}
Chan showed that (cf.~\cite[Eq.~(3.19)]{chan2005})
\begin{equation}\label{eq:Fq-bound}
F(|q|)\le \exp\left(\frac{e^{-2\pi \Im(y)}}{(1-e^{-2\pi \Im(y)})^2}\right).
\end{equation}

We see from \eqref{eq:gamma} that
$$\Im\left(\gamma_{(m',k')}(m\tau)\right)=\frac{1}{m'k'}\Re\left(\frac{1}{z}\right)\ge \frac{1}{m'k'}\frac{k}{2}=\frac{\gcd(m,k)^2}{2m}.$$
Hence the left-hand sides of \eqref{eq:bound-1} and \eqref{eq:bound-2} are both
$$\le \exp\left(\frac{e^{-2\pi \Im\left(\ee{\gamma_{(m',k')}(m\tau)}\right)}}{\left(1-e^{-2\pi \Im\left(\ee{\gamma_{(m',k')}(m\tau)}\right)}\right)^2}\right)\le \exp\left(\frac{e^{-\pi \gcd(m,k)^2/m}}{\left(1-e^{-\pi \gcd(m,k)^2/m}\right)^2}\right).$$
\end{proof}

We also need the following rough bound.
\begin{lemma}\label{le:rough-bound}
Let $\eta_1,\ldots,\eta_R$ be $R$ integers, all non-zero, and let $\alpha_1,\ldots,\alpha_R$ be $R$ distinct positive integers. Let $y\in\mathbb{H}$. Then
\begin{equation*}
\left|\prod_{r=1}^R F\left(\ee{\alpha_r y}\right)^{\eta_r}-1\right|\le \exp \left(\sum_{r=1}^R \frac{|\eta_r|e^{-2\pi \alpha_r\Im(y)}}{(1-e^{-2\pi \alpha_r\Im(y)})^2}\right)-1.
\end{equation*}
\end{lemma}

\begin{proof}
We begin with the observation that $\prod_{r=1}^R F(q^{\alpha_r})^{\eta_r}$ can be treated as the generating function of a weightd multi-colored partition. More precisely, for $F(q^{\alpha_r})$, the corresponding subpartition is an ordinary partition where we require parts being multiples of $\alpha_r$. For $F(q^{\alpha_r})^{-1}$, the corresponding subpartition is a weighted distinct partition with parts being multiples of $\alpha_r$, where the weight is given by $(-1)^\sharp$ with $\sharp$ counting the number of parts. Hence if we write (with $q=\ee{y}$)
$$\prod_{r=1}^R F\left(\ee{\alpha_r y}\right)^{\eta_r}=\sum_{n=0}^\infty a(n) q^n \quad\text{and}\quad \prod_{r=1}^R F\left(\ee{\alpha_r y}\right)^{|\eta_r|}=\sum_{n=0}^\infty \tilde{a}(n) q^n,$$
then $|a(n)|\le \tilde{a}(n)$ for $n\ge 0$.

Hence
\begin{align*}
\left|\prod_{r=1}^R F\left(\ee{\alpha_r y}\right)^{\eta_r}-1\right|&=\left|\sum_{n=0}^\infty a(n) q^n-1\right|=\left|\sum_{n=1}^\infty a(n) q^n\right|\\
&\le \sum_{n=1}^\infty \tilde{a}(n) |q|^n = \prod_{r=1}^R F\left(\left|q^{\alpha_r}\right|\right)^{|\eta_r|}-1\\
&\le \exp \left(\sum_{r=1}^R \frac{|\eta_r|e^{-2\pi \alpha_r\Im(y)}}{(1-e^{-2\pi \alpha_r\Im(y)})^2}\right)-1,
\end{align*}
where we use \eqref{eq:Fq-bound} in the last inequality.
\end{proof}
Since the modified Bessel functions of the first kind play an important role in our results, we establish the following bounds on $I_{0}(x)$.
\begin{lemma}\label{add-lem-Bessel}
For $x>0$ we have
\begin{align}\label{add-I-upper}
I_{0}(x)<\sqrt{\frac{\pi}{8}}\frac{e^x}{\sqrt{x}}.
\end{align}
For $x\ge 1$ we have
\begin{align}\label{add-I-lower}
I_{0}(x)>\frac{4\sqrt{2}}{5\pi}\frac{e^x}{\sqrt{x}}.
\end{align}
\end{lemma}
\begin{proof}
We need the following inequality: for $0\le t \le \frac{\pi}{2}$,
\begin{align}\label{sin-ineq}
\frac{2}{\pi}t \le \sin t \leq t.
\end{align}
For $x>0$, by \eqref{sin-ineq} we have
\begin{align*}
I_0(x) &= \frac{1}{\pi} \int_0^\pi e^{x \cos \theta}\ d\theta\\
&= \frac{2 e^x}{\pi} \int_0^{\pi/2} e^{-2 x \sin^2 u}\ d u\\
&\le \frac{2 e^x}{\pi} \int_0^{\pi/2} e^{- \frac{8 x}{\pi^2} u^2}\ du\\
& <\frac{2 e^x}{\pi} \frac{\pi}{\sqrt{8 x}} \int_0^\infty e^{-t^2}\ dt\\
& = \sqrt{\frac{\pi}{8}} \frac{e^x}{\sqrt{x}}.
\end{align*}
Similarly, for $x\ge 1$, by \eqref{sin-ineq} we have
\begin{align*}
I_0(x)&= \frac{2 e^x}{\pi} \int_0^{\pi/2} e^{-2 x \sin^2 u}\ d u\\
&\ge  \frac{2 e^x}{\pi} \int_0^{\pi/2} e^{-2 x  u^2}\ d u\\
&\ge \frac{\sqrt{2}}{\pi}\frac{e^x}{\sqrt{x}}\int_{0}^{\frac{\pi\sqrt{x}}{\sqrt{2}}}e^{-t^2}dt  \\
&\ge \frac{\sqrt{2}}{\pi}\frac{e^x}{\sqrt{x}}\int_{0}^{\frac{\pi}{\sqrt{2}}}e^{-t^2}dt  \\
&>\frac{4\sqrt{2}}{5\pi}\frac{e^x}{\sqrt{x}}. \qedhere
\end{align*}
\end{proof}
\begin{remark}
The inequality \eqref{add-I-upper} will be used many times in our arguments while \eqref{add-I-lower} will not be used. The purpose for giving \eqref{add-I-lower} is to help the readers understand roughly how the main term dominates the error term in Theorems \ref{th:main-asy} and \ref{th:main-asy-b}.
\end{remark}
\section{Asymptotic behavior of $M^{*}(0,3,n)-M^{*}(1,3,n)$}\label{asymptotic-3}

For notational convenience, $u\equiv_m v $ means $u\equiv v \pmod{m}$, and likewise $u\not\equiv_m v $ means $u\not\equiv v \pmod{m}$.

Taking $z=\zeta_{3}=e^{2\pi i/3}$ in \eqref{bicrank:gf} yields
\begin{align}
 &\sum_{n=0}^{\infty}\left(M^{*}(0,3,n)-M^{*}(1,3,n)\right)q^{n}=\dfrac{(q;q)_{\infty}^{2}}{(\zeta_{3}q,\zeta_{3}^{-1}q,\zeta_{3}^{2}q,\zeta_{3}^{-2}q;q)_{\infty}}
 =\dfrac{(q;q)_{\infty}^{4}}{(q^{3};q^{3})_{\infty}^{2}}.\label{eq:gen-mod-3}
\end{align}
For convenience, we write
\begin{equation*}
a(n)=M^{*}(0,3,n)-M^{*}(1,3,n).
\end{equation*}
Let
\begin{equation*}
f\left(\ee{\tau}\right)=\sum_{n=0}^\infty a(n) q^n = \frac{\f{1}^4}{\f{3}^2}=\frac{F\left(\ee{3\tau}\right)^2}{F\left(\ee{\tau}\right)^4}.
\end{equation*}
It follows from \eqref{eq:cauchy-var} that
\begin{align}
a(n)&=\sum_{1\le k\le N} \sum_{\substack{0\le h\le k\\ (h,k)=1}} \ee{-\frac{nh}{k}} \int_{\xi_{h,k}} f\big(\ee{\tau}\big)\ee{-n \phi} e^{2 \pi n \varrho}\ d\phi\nonumber\\
&=\left(\sum_{\substack{1\le k\le N\\k\not\equiv_3 0}}+\sum_{\substack{1\le k\le N\\k\equiv_3 0}}\right) \sum_{\substack{0\le h\le k\\ (h,k)=1}} \ee{-\frac{nh}{k}} \int_{\xi_{h,k}} f\big(\ee{\tau}\big)\ee{-n \phi} e^{2 \pi n \varrho}\ d\phi\nonumber\\
&=:S_1+S_2.
\end{align}

\subsection{Estimate of $S_1$}

When $k\not\equiv 0 \pmod{3}$, we have
\begin{align*}
f\big(\ee{\tau}\big)& =\frac{F\left(\ee{3\tau}\right)^2}{F\left(\ee{\tau}\right)^4}\\
&= \frac{F\left(\ee{\gamma_{(3,k)}\left(3\tau\right)}\right)^2}{F\left(\ee{\gamma_{(1,k)}\left(\tau\right)}\right)^4}e^{\frac{\pi}{12 k}\left(\frac{2}{3z}-6z-\frac{4}{z}+4z\right)}\ee{\frac{2s(3h,k)-4s(h,k)}{2}}\left(3z^{-1}\right),
\end{align*}
where the last equality comes from \eqref{eq:F-m-gamma}. Hence it follows from \eqref{eq:Re-z-bound}--\eqref{eq:bound-2} that
\begin{align*}
&\left|f\big(\ee{\tau}\big)\ee{-n \phi} e^{2 \pi n \varrho}\right|\\
&\quad= \left|\frac{F\left(\ee{\gamma_{(3,k)}\left(3\tau\right)}\right)^2}{F\left(\ee{\gamma_{(1,k)}\left(\tau\right)}\right)^4}\right| \left|e^{-\frac{5\pi}{18 k}\frac{1}{z}-\frac{\pi}{6k}z}\right| \left(3|z|^{-1}\right)e^{2\pi n \varrho}\\
&\quad \le 3 k^{-1} N^2 e^{2\pi\varrho \left(n-\frac{1}{12}\right)} \exp\left(-\frac{5\pi}{36}+\frac{2e^{-\pi/3}}{\left(1-e^{-\pi/3}\right)^2}+\frac{4e^{-\pi}}{\left(1-e^{-\pi}\right)^2}\right).
\end{align*}

We therefore have
\begin{align*}
|S_1|&=\left|\sum_{\substack{1\le k\le N\\k\not\equiv_3 0}}\sum_{\substack{0\le h\le k\\ (h,k)=1}} \ee{-\frac{nh}{k}} \int_{\xi_{h,k}} f\big(\ee{\tau}\big)\ee{-n \phi} e^{2 \pi n \varrho}\ d\phi\right|\\
& \le \sum_{\substack{1\le k\le N\\k\not\equiv_3 0}}\sum_{\substack{0\le h\le k\\ (h,k)=1}} \int_{\xi_{h,k}} 3 k^{-1} N^2 e^{2\pi\varrho \left(n-\frac{1}{12}\right)}\\
&\quad\quad\quad\quad\quad\quad\quad\quad\times \exp\left(-\frac{5\pi}{36}+\frac{2e^{-\pi/3}}{\left(1-e^{-\pi/3}\right)^2}+\frac{4e^{-\pi}}{\left(1-e^{-\pi}\right)^2}\right) \ d\phi\\
&\le \sum_{\substack{1\le k\le N\\k\not\equiv_3 0}}\sum_{\substack{0\le h\le k\\ (h,k)=1}} 3 k^{-1} N^2 e^{2\pi\varrho \left(n-\frac{1}{12}\right)}\\
&\quad\quad\quad\quad\quad\quad\quad\quad\times \exp\left(-\frac{5\pi}{36}+\frac{2e^{-\pi/3}}{\left(1-e^{-\pi/3}\right)^2}+\frac{4e^{-\pi}}{\left(1-e^{-\pi}\right)^2}\right)\frac{2}{kN} \tag{by \eqref{eq:xi-bound}}\\
&\le \sum_{\substack{1\le k\le N\\k\not\equiv_3 0}} 3 k^{-1} N^2 e^{2\pi\varrho \left(n-\frac{1}{12}\right)} \exp\left(-\frac{5\pi}{36}+\frac{2e^{-\pi/3}}{\left(1-e^{-\pi/3}\right)^2}+\frac{4e^{-\pi}}{\left(1-e^{-\pi}\right)^2}\right)\frac{2}{kN}\ k\\
&\le 6\; e^{2\pi\varrho \left(n-\frac{1}{12}\right)} N\big(\log N+1\big) \exp\left(-\frac{5\pi}{36}+\frac{2e^{-\pi/3}}{\left(1-e^{-\pi/3}\right)^2}+\frac{4e^{-\pi}}{\left(1-e^{-\pi}\right)^2}\right)\\
&\le 24.8\; e^{2\pi\varrho \left(n-\frac{1}{12}\right)} N\big(\log N+1\big).
\end{align*}

\subsection{Estimate of $S_2$}

When $k\equiv 0 \pmod{3}$, we write $k=3k'$. Define
$$\omega_{h,k'}:=\ee{\frac{2s(h,k')-4s(h,3k')}{2}}.$$
We also write for convenience
$$\tau':=\gamma_{(1,3k')}(\tau)=\frac{h'+iz^{-1}}{3k'}.$$
Then we know from \eqref{eq:gamma} that
$$\gamma_{(1,k')}\left(3\tau\right)=\frac{h'+iz^{-1}}{k'}=3\tau'.$$
It also follows from \eqref{eq:Re-1-z-bound} that
$$\Im(\tau')=\frac{1}{3k'}\Re\left(\frac{1}{z}\right)\ge \frac{1}{3k'} \frac{3k'}{2}=\frac{1}{2}.$$

We deduce from \eqref{eq:F-m-gamma} that
\begin{align*}
f\big(\ee{\tau}\big)& =\frac{F\left(\ee{3\tau}\right)^2}{F\left(\ee{\tau}\right)^4}\\
&= \frac{F\left(\ee{\gamma_{(1,k')}\left(3\tau\right)}\right)^2}{F\left(\ee{\gamma_{(1,3k')}\left(\tau\right)}\right)^4}\ \omega_{h,k'}\ e^{\frac{\pi}{18k'}\left(\frac{1}{z}-z\right)}\ z^{-1}\\
&=\frac{F\left(\ee{3\tau'}\right)^2}{F\left(\ee{\tau'}\right)^4}\ \omega_{h,k'}\ e^{\frac{\pi}{18k'}\left(\frac{1}{z}-z\right)}\ z^{-1}.
\end{align*}
Then
\begin{align}
S_2&=\sum_{\substack{1\le k\le N\\k\equiv_3 0}}\sum_{\substack{0\le h\le k\\ (h,k)=1}} \ee{-\frac{nh}{k}} \int_{\xi_{h,k}} f\big(\ee{\tau}\big)\ee{-n \phi} e^{2 \pi n \varrho}\ d\phi \nonumber\\
& = \sum_{1\le 3k'\le N}\sum_{\substack{0\le h\le 3k'\\ (h,3k')=1}} \ee{-\frac{nh}{3k'}} \int_{\xi_{h,3k'}} \frac{F\left(\ee{3\tau'}\right)^2}{F\left(\ee{\tau'}\right)^4} \nonumber\\
&\quad\quad\quad\quad\quad\quad\quad\quad\quad\quad\quad\quad\quad\quad\quad \times\omega_{h,k'}\ e^{\frac{\pi}{18k'}\left(\frac{1}{z}-z\right)}\ z^{-1} \ee{-n \phi} e^{2 \pi n \varrho}\ d\phi \nonumber\\
&= \sum_{1\le 3k'\le N}\sum_{\substack{0\le h\le 3k'\\ (h,3k')=1}} \ee{-\frac{nh}{3k'}} \int_{\xi_{h,3k'}} \omega_{h,k'}\ e^{\frac{\pi}{18k'}\left(\frac{1}{z}-z\right)}\ z^{-1} \ee{-n \phi} e^{2 \pi n \varrho}\ d\phi \nonumber\\
&\quad + \sum_{1\le 3k'\le N}\sum_{\substack{0\le h\le 3k'\\ (h,3k')=1}} \ee{-\frac{nh}{3k'}} \int_{\xi_{h,3k'}} \left(\frac{F\left(\ee{3\tau'}\right)^2}{F\left(\ee{\tau'}\right)^4}-1\right) \nonumber\\
&\quad\quad\quad\quad\quad\quad\quad\quad\quad\quad\quad\quad\quad\quad\quad\quad \times\omega_{h,k'}\ e^{\frac{\pi}{18k'}\left(\frac{1}{z}-z\right)}\ z^{-1} \ee{-n \phi} e^{2 \pi n \varrho}\ d\phi \nonumber\\
&=:T_1+T_2.\label{eq:a-S2}
\end{align}

We first estimate $T_2$. It follows from Lemma \ref{le:rough-bound} that
\begin{align*}
\left|\frac{F\left(\ee{3\tau'}\right)^2}{F\left(\ee{\tau'}\right)^4}-1\right|\le \exp\left(\frac{2e^{-6\pi \Im(\tau')}}{\left(1-e^{-6\pi \Im(\tau')}\right)^2}+\frac{4e^{-2\pi \Im(\tau')}}{\left(1-e^{-2\pi \Im(\tau')}\right)^2}\right)-1.
\end{align*}
Hence
\begin{align*}
&\left|\left(\frac{F\left(\ee{3\tau'}\right)^2}{F\left(\ee{\tau'}\right)^4}-1\right)e^{\frac{\pi}{18k'}\frac{1}{z}}\right|=e^{\frac{\pi}{6}\Im(\tau')} \left|\frac{F\left(\ee{3\tau'}\right)^2}{F\left(\ee{\tau'}\right)^4}-1\right|\\
&\quad \le e^{\frac{\pi}{6}\Im(\tau')}  \left(\exp\left(\frac{2e^{-6\pi \Im(\tau')}}{\left(1-e^{-6\pi \Im(\tau')}\right)^2}+\frac{4e^{-2\pi \Im(\tau')}}{\left(1-e^{-2\pi \Im(\tau')}\right)^2}\right)-1\right)\\
&\quad \le e^{\frac{\pi}{12}} \left(\exp\left(\frac{2e^{-3\pi}}{\left(1-e^{-3\pi }\right)^2}+\frac{4e^{-\pi }}{\left(1-e^{-\pi }\right)^2}\right)-1\right),
\end{align*}
since the second line is indeed a decreasing function of $\Im(\tau')$ for $\Im(\tau')\ge 1/2$. This yields
\begin{align*}
|T_2|&\le \sum_{1\le 3k'\le N}\sum_{\substack{0\le h\le 3k'\\ (h,3k')=1}}  \int_{\xi_{h,3k'}} e^{2\pi\varrho \left(n-\frac{1}{12}\right)}\ \frac{N^2}{3k'}\ e^{\frac{\pi}{12}}\\
&\quad\quad\quad\quad\quad\quad\quad\quad\quad\times \left(\exp\left(\frac{2e^{-3\pi}}{\left(1-e^{-3\pi }\right)^2}+\frac{4e^{-\pi }}{\left(1-e^{-\pi }\right)^2}\right)-1\right)\ d\phi\\
&\le \sum_{1\le 3k'\le N}\sum_{\substack{0\le h\le 3k'\\ (h,3k')=1}}  e^{2\pi\varrho \left(n-\frac{1}{12}\right)}\ \frac{N^2}{3k'}\ e^{\frac{\pi}{12}}\\
&\quad\quad\quad\quad\quad\quad\quad\quad\quad\times \left(\exp\left(\frac{2e^{-3\pi}}{\left(1-e^{-3\pi }\right)^2}+\frac{4e^{-\pi }}{\left(1-e^{-\pi }\right)^2}\right)-1\right)\frac{2}{3k'N}\\
&\le \sum_{1\le 3k'\le N} e^{2\pi\varrho \left(n-\frac{1}{12}\right)}\ \frac{N^2}{3k'}\ e^{\frac{\pi}{12}}\\
&\quad\quad\quad\quad\quad\quad\quad\quad\quad\times \left(\exp\left(\frac{2e^{-3\pi}}{\left(1-e^{-3\pi }\right)^2}+\frac{4e^{-\pi }}{\left(1-e^{-\pi }\right)^2}\right)-1\right)\frac{2}{3k'N}\ 3k'\\
&\le 2\;e^{2\pi\varrho \left(n-\frac{1}{12}\right)} N\big(\log N+1\big) e^{\frac{\pi}{12}} \left(\exp\left(\frac{2e^{-3\pi}}{\left(1-e^{-3\pi }\right)^2}+\frac{4e^{-\pi }}{\left(1-e^{-\pi }\right)^2}\right)-1\right)\\
&\le 0.6\;e^{2\pi\varrho \left(n-\frac{1}{12}\right)} N\big(\log N+1\big).
\end{align*}

We now estimate $T_1$, which gives the main contribution. The following evaluation of an integral plays an important role.

\begin{lemma}\label{le:key-int}
Let $\gcd(h,k)=1$ and define
\begin{equation}
I:=\int_{\xi_{h,k}} e^{\frac{\pi}{6k}\left(\frac{1}{z}-z\right)}z^{-1}\ee{-n\phi}e^{2\pi n \varrho}\ d\phi.
\end{equation}
Then
\begin{equation}
I=\frac{2\pi}{k}\; I_0\left(\frac{2\pi \sqrt{n-\frac{1}{12}}}{\sqrt{3}k}\right)+E(I)
\end{equation}
with
\begin{equation}
|E(I)|\le5.2\; \frac{e^{2\pi\varrho \left(n-\frac{1}{12}\right)}N}{n-\frac{1}{12}}.
\end{equation}
\end{lemma}

\begin{proof}
Making the following change of variables $w=z/k=\varrho-i\phi$ yields
$$I=\frac{1}{2\pi i}\int_{\varrho-i\theta''_{h,k}}^{\varrho+i\theta'_{h,k}} 2\pi e^{\frac{\pi}{6k^2 w}} e^{2\pi w \left(n-\frac{1}{12}\right)}(kw)^{-1}\ dw.$$
We may further separate the integral as
\begin{align*}
I&=\frac{1}{2\pi i}\left(\int_\Gamma-\int_{-\infty-i\theta''_{h,k}}^{\varrho-i\theta''_{h,k}}+\int_{-\infty+i\theta'_{h,k}}^{\varrho+i\theta'_{h,k}}\right) 2\pi e^{\frac{\pi}{6k^2 w}} e^{2\pi w \left(n-\frac{1}{12}\right)}(kw)^{-1}\ dw\\
&=:J_1-J_2+J_3,
\end{align*}
where
\begin{align*}
\Gamma:=(-\infty-i\theta''_{h,k}) \to (\varrho-i\theta''_{h,k}) \to (\varrho+i\theta'_{h,k}) \to (-\infty+i\theta'_{h,k})
\end{align*}
is a Hankel contour.

To bound $J_2$ and $J_3$, which contribute to $E(I)$, we write $w=x+i\theta$ with $-\infty\le x\le \varrho$ and $\theta\in\{\theta'_{h,k},-\theta''_{h,k}\}$. Following Chan \cite[p.~120]{chan2005}, we have (noting that $\frac{1}{2kN}\le |\theta|\le \frac{1}{kN}$)
\begin{align*}
\left|e^{\frac{\pi}{6k^2 w}}\right|&=e^{\frac{\pi}{6k^2}\Re\left(\frac{1}{w}\right)}=e^{\frac{\pi}{6k^2}\frac{x}{x^2+\theta^2}}\\
&\le e^{\frac{\pi}{6k^2}\frac{x}{\theta^2}}\le e^{\frac{\pi}{6k^2}\varrho (2kN)^2}=e^{\frac{2\pi}{3}},\\
\left|e^{2\pi w \left(n-\frac{1}{12}\right)}\right|&\le e^{2\pi x \left(n-\frac{1}{12}\right)},\\
\left|(kw)^{-1}\right|&\le \frac{1}{k\sqrt{x^2+\theta^2}}\le \frac{1}{k|\theta|}\le 2N.
\end{align*}
Hence for $j=2$ and $3$, we have
\begin{align*}
|J_j|&\le \frac{1}{2\pi} \int_{-\infty}^{\varrho} 2\pi e^{\frac{2\pi}{3}}e^{2\pi x \left(n-\frac{1}{12}\right)}\ 2N\ dx\\
&=\frac{e^{\frac{2\pi}{3}} N}{\pi \left(n-\frac{1}{12}\right)}e^{2\pi \varrho \left(n-\frac{1}{12}\right)}\le 2.6 \frac{N}{n-\frac{1}{12}}e^{2\pi \varrho \left(n-\frac{1}{12}\right)}.
\end{align*}
This implies that
$$|-J_2+J_3|\le |J_2|+|J_3|\le 5.2\; \frac{e^{2\pi\varrho \left(n-\frac{1}{12}\right)}N}{n-\frac{1}{12}}.$$

We now compute the main term $J_1$. Making the following change of variables $t=wk\sqrt{12n-1}$ yields
$$J_1=\frac{2\pi}{k}\frac{1}{2\pi i}\int_{\tilde{\Gamma}} e^{\frac{\pi \sqrt{n-\frac{1}{12}}}{\sqrt{3}k}\left(t+\frac{1}{t}\right)} t^{-1}\  dt.$$
Note that the new contour $\tilde{\Gamma}$ is still a Hankel contour. From the contour integral representation of $I_s(x)$ (cf.~\cite[pp.~610--616]{Arf1985})
$$I_s(x)=\frac{1}{2\pi i}\int_{\Gamma} t^{-s-1}e^{\frac{x}{2}\left(t+\frac{1}{t}\right)}\ dt\quad\text{($\Gamma$ is a Hankel contour)},$$
we conclude
$$J_1=\frac{2\pi}{k}\; I_0\left(\frac{2\pi \sqrt{n-\frac{1}{12}}}{\sqrt{3}k}\right),$$
and therefore complete our proof.
\end{proof}

It follows from Lemma \ref{le:key-int} with $k=3k'$ that
\begin{align*}
T_1&=\sum_{1\le 3k'\le N}\sum_{\substack{0\le h\le 3k'\\ (h,3k')=1}} \ee{-\frac{nh}{3k'}} \int_{\xi_{h,3k'}} \omega_{h,k'}\ e^{\frac{\pi}{18k'}\left(\frac{1}{z}-z\right)}\ z^{-1} \ee{-n \phi} e^{2 \pi n \varrho}\ d\phi\\
&=D_n+\sum_{1\le 3k'\le N}\sum_{\substack{0\le h\le 3k'\\ (h,3k')=1}} \ee{-\frac{nh}{3k'}} \omega_{h,k'} \frac{2\pi}{3k'}\; I_0\left(\frac{2\pi \sqrt{n-\frac{1}{12}}}{3\sqrt{3}k'}\right),
\end{align*}
with
\begin{align*}
|D_n|\le \sum_{1\le 3k'\le N}\sum_{\substack{0\le h\le 3k'\\ (h,3k')=1}} 5.2\; \frac{e^{2\pi\varrho \left(n-\frac{1}{12}\right)}N}{n-\frac{1}{12}} \le \frac{5.2}{3}\; \frac{e^{2\pi\varrho \left(n-\frac{1}{12}\right)}N^3}{n-\frac{1}{12}}.
\end{align*}

For the main term in $T_1$, we compute the $k'=1$ case
\begin{align*}
\sum_{\substack{0\le h\le 3\\ (h,3)=1}} \ee{-\frac{nh}{3}} \omega_{h,1} \frac{2\pi}{3}&=\begin{cases}
\frac{4\pi}{3} \cos\frac{2\pi}{9} & \text{if $n\equiv 0 \pmod{3}$},\\
-\frac{4\pi}{3} \cos\frac{\pi}{9} & \text{if $n\equiv 1 \pmod{3}$},\\
\frac{4\pi}{3} \sin\frac{\pi}{18} & \text{if $n\equiv 2 \pmod{3}$}
\end{cases}\\
&=:c(n).
\end{align*}
Hence we obtain the main term in \eqref{eq:main-asy-a-1}
$$c(n) \; I_0\left(\frac{2\pi \sqrt{n-\frac{1}{12}}}{3\sqrt{3}}\right).$$

For $2\le k'\le N/3$, using \eqref{add-I-upper} we deduce that
\begin{align*}
|H_n|:=&\left|\sum_{2\le k'\le N/3}\sum_{\substack{0\le h\le 3k'\\ (h,3k')=1}} \ee{-\frac{nh}{3k'}} \omega_{h,k'} \frac{2\pi}{3k'}\; I_0\left(\frac{2\pi \sqrt{n-\frac{1}{12}}}{3\sqrt{3}k'}\right)\right|\\
\le& \sum_{2\le k'\le N/3} 3k' \frac{2\pi}{3k'} \sqrt{\frac{\pi}{8}} \frac{e^{\frac{2\pi \sqrt{n-\frac{1}{12}}}{3\sqrt{3}k'}}}{\sqrt{\frac{2\pi \sqrt{n-\frac{1}{12}}}{3\sqrt{3}k'}}}\\
\le& \frac{3^{-\frac{3}{4}}\pi}{2}N^{\frac{3}{2}} \left(n-\frac{1}{12}\right)^{-\frac{1}{4}}  e^{\frac{\pi \sqrt{n-\frac{1}{12}}}{3\sqrt{3}}}\\
\le& 0.7\;N^{\frac{3}{2}} \left(n-\frac{1}{12}\right)^{-\frac{1}{4}}  e^{\frac{\pi \sqrt{n-\frac{1}{12}}}{3\sqrt{3}}},
\end{align*}
where we use the trivial bound $\sum\limits_{n=1}^{N/3} n^{1/2}\le 3^{-3/2}N^{3/2}$ for $N\ge 1$.

\subsection{The asymptotic formula of $a(n)$}

It follows that the total error term is
\begin{align*}
|E(n)|&\le|S_1|+|T_2|+|D_n|+|H_n|\\
& \le 25.4\; e^{2\pi\varrho \left(n-\frac{1}{12}\right)} N\big(\log N+1\big)+\frac{5.2}{3}\; \frac{e^{2\pi\varrho \left(n-\frac{1}{12}\right)}N^3}{n-\frac{1}{12}}\\
&\quad +0.7\;N^{\frac{3}{2}} \left(n-\frac{1}{12}\right)^{-\frac{1}{4}}  e^{\frac{\pi \sqrt{n-\frac{1}{12}}}{3\sqrt{3}}}.
\end{align*}
Setting $N=\sqrt{2\pi (n-1/12)}$ yields
\begin{align*}
|E(n)|& \le 25.4\; e \sqrt{2\pi} \sqrt{n-\frac{1}{12}}\; \left(\log \sqrt{2\pi \left(n-\frac{1}{12}\right)}+1\right)\\
&\quad+\frac{5.2\;e}{3} \left(\sqrt{2\pi}\right)^3\sqrt{n-\frac{1}{12}}\\
&\quad +0.7\;\left(\sqrt{2\pi}\right)^{\frac{3}{2}} \sqrt{n-\frac{1}{12}} \; e^{\frac{\pi \sqrt{n-\frac{1}{12}}}{3\sqrt{3}}}\\
&\le 173.1\; \sqrt{n-\frac{1}{12}}\; \left(\log \sqrt{2\pi \left(n-\frac{1}{12}\right)}+1\right) + 74.3\; \sqrt{n-\frac{1}{12}}\\
&\quad+ 2.8\; \sqrt{n-\frac{1}{12}} \; e^{\frac{\pi \sqrt{n-\frac{1}{12}}}{3\sqrt{3}}}.
\end{align*}

Together with the main term
$$c(n) \; I_0\left(\frac{2\pi \sqrt{n-\frac{1}{12}}}{3\sqrt{3}}\right),$$
we arrive at the desired asymptotic formula.

\subsection{Proof of Theorem \ref{THM:mod 3}}

We deduce from a short computation that the sign of $M^{*}(0,3,n)-M^{*}(1,3,n)$ is determined by the main term (and hence by $c(n)$) when $n\ge 114$. For the remaining cases, we check them directly and hence arrive at the desired inequalities.

\section{Asymptotic behavior of $M^{*}(0,4,n)-M^{*}(2,4,n)$}\label{asymptotic-4}

This time we take $z=i$ in \eqref{bicrank:gf} to obtain
\begin{align}
\sum_{n=0}^{\infty}\left(M^{*}(0,4,n)-M^{*}(2,4,n)\right)q^{n} &=\dfrac{(q;q)_{\infty}^{2}}{(iq,-iq,-q,-q;q)_{\infty}}=\dfrac{(q;q)_{\infty}^{2}}{(-q^{2};q^{2})_{\infty}
(-q;q)_{\infty}^{2}}\nonumber\\
 &=\dfrac{(q;q)_{\infty}^{4}}{(q^{2};q^{2})_{\infty}(q^{4};q^{4})_{\infty}}.\label{eq:gen-mod-4}
\end{align}

Again, we write for convenience
\begin{equation*}
b(n)=M^{*}(0,4,n)-M^{*}(2,4,n).
\end{equation*}
We will only outline the proof of the asymptotic formula of $b(n)$ as it has the same manner as the previous proof.

Let
\begin{equation*}
f\left(\ee{\tau}\right)=\sum_{n=0}^\infty b(n) q^n = \frac{\f{1}^4}{\f{2} \f{4}}=\frac{F\left(\ee{2\tau}\right)F\left(\ee{4\tau}\right)}{F\left(\ee{\tau}\right)^4}.
\end{equation*}
It follows from \eqref{eq:cauchy-var} that
\begin{align*}
b(n)&=\sum_{1\le k\le N} \sum_{\substack{0\le h\le k\\ (h,k)=1}} \ee{-\frac{nh}{k}} \int_{\xi_{h,k}} f\big(\ee{\tau}\big)\ee{-n \phi} e^{2 \pi n \varrho}\ d\phi\nonumber\\
&=\left(\sum_{\substack{1\le k\le N\\k\equiv_2 1}}+\sum_{\substack{1\le k\le N\\k\equiv_4 2}}+\sum_{\substack{1\le k\le N\\k\equiv_4 0}}\right) \sum_{\substack{0\le h\le k\\ (h,k)=1}} \ee{-\frac{nh}{k}} \int_{\xi_{h,k}} f\big(\ee{\tau}\big)\ee{-n \phi} e^{2 \pi n \varrho} d\phi\nonumber\\
&=:S_1+S_2+S_3.
\end{align*}

\subsection{Estimate of $S_1$}

We have
\begin{align*}
|S_1|&=\left|\sum_{\substack{1\le k\le N\\k\equiv_2 1}}\sum_{\substack{0\le h\le k\\ (h,k)=1}} \ee{-\frac{nh}{k}} \int_{\xi_{h,k}} f\big(\ee{\tau}\big)\ee{-n \phi} e^{2 \pi n \varrho}\ d\phi\right|\\
&\le \sum_{\substack{1\le k\le N\\k\equiv_2 1}}\sum_{\substack{0\le h\le k\\ (h,k)=1}} \int_{\xi_{h,k}} \left|\frac{F\left(\ee{\gamma_{(2,k)}\left(2\tau\right)}\right) F\left(\ee{\gamma_{(4,k)}\left(4\tau\right)}\right)}{F\left(\ee{\gamma_{(1,k)}\left(\tau\right)}\right)^4}\right|\\
&\quad\quad\quad\quad\quad\quad\quad\quad\times \left|e^{-\frac{13\pi}{48 k}\frac{1}{z}-\frac{\pi}{6k}z}\right| \left(2\sqrt{2}|z|^{-1}\right)e^{2\pi n \varrho}\ d\phi\\
& \le \sum_{\substack{1\le k\le N\\k\equiv_2 1}}\sum_{\substack{0\le h\le k\\ (h,k)=1}} \int_{\xi_{h,k}} 2\sqrt{2} k^{-1} N^2 e^{2\pi\varrho \left(n-\frac{1}{12}\right)}\\
&\quad\quad\quad\quad\times \exp\left(-\frac{13\pi}{96}+\frac{e^{-\pi/2}}{\left(1-e^{-\pi/2}\right)^2}+\frac{e^{-\pi/4}}{\left(1-e^{-\pi/4}\right)^2}+\frac{4e^{-\pi}}{\left(1-e^{-\pi}\right)^2}\right) \ d\phi\\
&\le \sum_{\substack{1\le k\le N\\k\equiv_2 1}}\sum_{\substack{0\le h\le k\\ (h,k)=1}} 2\sqrt{2} k^{-1} N^2 e^{2\pi\varrho \left(n-\frac{1}{12}\right)}\\
&\quad\quad\quad\quad\times \exp\left(-\frac{13\pi}{96}+\frac{e^{-\pi/2}}{\left(1-e^{-\pi/2}\right)^2}+\frac{e^{-\pi/4}}{\left(1-e^{-\pi/4}\right)^2}+\frac{4e^{-\pi}}{\left(1-e^{-\pi}\right)^2}\right) \frac{2}{kN}\\
&\le \sum_{\substack{1\le k\le N\\k\equiv_2 1}} 2\sqrt{2} k^{-1} N^2 e^{2\pi\varrho \left(n-\frac{1}{12}\right)}\\
&\quad\quad\quad\quad\times \exp\left(-\frac{13\pi}{96}+\frac{e^{-\pi/2}}{\left(1-e^{-\pi/2}\right)^2}+\frac{e^{-\pi/4}}{\left(1-e^{-\pi/4}\right)^2}+\frac{4e^{-\pi}}{\left(1-e^{-\pi}\right)^2}\right) \frac{2}{kN}\ k\\
&\le 4\sqrt{2}\; e^{2\pi\varrho \left(n-\frac{1}{12}\right)} N\big(\log N+1\big)\\
&\quad\quad\quad\quad\times \exp\left(-\frac{13\pi}{96}+\frac{e^{-\pi/2}}{\left(1-e^{-\pi/2}\right)^2}+\frac{e^{-\pi/4}}{\left(1-e^{-\pi/4}\right)^2}+\frac{4e^{-\pi}}{\left(1-e^{-\pi}\right)^2}\right)\\
&\le 29.1\; e^{2\pi\varrho \left(n-\frac{1}{12}\right)} N\big(\log N+1\big).
\end{align*}

\subsection{Estimate of $S_2$}

Let $k=2k'$ with $k'\equiv 1 \pmod{2}$. We have
\begin{align*}
|S_2|&=\left|\sum_{\substack{1\le k\le N\\k\equiv_4 2}}\sum_{\substack{0\le h\le k\\ (h,k)=1}} \ee{-\frac{nh}{k}} \int_{\xi_{h,k}} f\big(\ee{\tau}\big)\ee{-n \phi} e^{2 \pi n \varrho}\ d\phi\right|\\
&\le \sum_{\substack{1\le 2k'\le N\\k'\equiv_2 1}}\sum_{\substack{0\le h\le 2k'\\ (h,2k')=1}} \int_{\xi_{h,2k'}} \left|\frac{F\left(\ee{\gamma_{(1,k')}\left(2\tau\right)}\right) F\left(\ee{\gamma_{(2,k')}\left(4\tau\right)}\right)}{F\left(\ee{\gamma_{(1,2k')}\left(\tau\right)}\right)^4}\right|\\
&\quad\quad\quad\quad\quad\quad\quad\quad\times \left|e^{-\frac{\pi}{24 k'}\frac{1}{z}-\frac{\pi}{12k'}z}\right| \left(\sqrt{2}|z|^{-1}\right)e^{2\pi n \varrho}\ d\phi\\
& \le \sum_{\substack{1\le 2k'\le N\\k'\equiv_2 1}}\sum_{\substack{0\le h\le 2k'\\ (h,2k')=1}} \int_{\xi_{h,2k'}} \sqrt{2} \left(2k'\right)^{-1} N^2 e^{2\pi\varrho \left(n-\frac{1}{12}\right)}\\
&\quad\quad\quad\quad\times \exp\left(-\frac{\pi}{24}+\frac{e^{-2\pi}}{\left(1-e^{-2\pi}\right)^2}+\frac{e^{-\pi}}{\left(1-e^{-\pi}\right)^2}+\frac{4e^{-\pi}}{\left(1-e^{-\pi}\right)^2}\right) \ d\phi\\
&\le \sum_{\substack{1\le 2k'\le N\\k'\equiv_2 1}}\sum_{\substack{0\le h\le 2k'\\ (h,2k')=1}} \sqrt{2} \left(2k'\right)^{-1} N^2 e^{2\pi\varrho \left(n-\frac{1}{12}\right)}\\
&\quad\quad\quad\quad\times \exp\left(-\frac{\pi}{24}+\frac{e^{-2\pi}}{\left(1-e^{-2\pi}\right)^2}+\frac{e^{-\pi}}{\left(1-e^{-\pi}\right)^2}+\frac{4e^{-\pi}}{\left(1-e^{-\pi}\right)^2}\right) \frac{2}{2k'N}\\
&\le \sum_{\substack{1\le 2k'\le N\\k'\equiv_2 1}} \sqrt{2} \left(2k'\right)^{-1} N^2 e^{2\pi\varrho \left(n-\frac{1}{12}\right)}\\
&\quad\quad\quad\quad\times \exp\left(-\frac{\pi}{24}+\frac{e^{-2\pi}}{\left(1-e^{-2\pi}\right)^2}+\frac{e^{-\pi}}{\left(1-e^{-\pi}\right)^2}+\frac{4e^{-\pi}}{\left(1-e^{-\pi}\right)^2}\right) \frac{2}{2k'N}\ 2k'\\
&\le 2\sqrt{2}\; e^{2\pi\varrho \left(n-\frac{1}{12}\right)} N\big(\log N+1\big)\\
&\quad\quad\quad\quad\times \exp\left(-\frac{\pi}{24}+\frac{e^{-2\pi}}{\left(1-e^{-2\pi}\right)^2}+\frac{e^{-\pi}}{\left(1-e^{-\pi}\right)^2}+\frac{4e^{-\pi}}{\left(1-e^{-\pi}\right)^2}\right)\\
&\le 3.2\; e^{2\pi\varrho \left(n-\frac{1}{12}\right)} N\big(\log N+1\big).
\end{align*}

\subsection{Estimate of $S_3$}

Let $k=4k'$. Define
$$\omega_{h,k'}:=\ee{\frac{s(h,2k')+s(h,k')-4s(h,4k')}{2}}.$$
We also write for convenience
$$\tau':=\gamma_{(1,4k')}(\tau)=\frac{h'+iz^{-1}}{4k'}.$$
Again, we have $\Im(\tau')\ge 1/2$.

Similar to \eqref{eq:a-S2}, we separate $S_3$ into two parts:
\begin{align*}
S_3&=\sum_{\substack{1\le k\le N\\k\equiv_4 0}}\sum_{\substack{0\le h\le k\\ (h,k)=1}} \ee{-\frac{nh}{k}} \int_{\xi_{h,k}} f\big(\ee{\tau}\big)\ee{-n \phi} e^{2 \pi n \varrho}\ d\phi \nonumber\\
& = \sum_{1\le 4k'\le N}\sum_{\substack{0\le h\le 4k'\\ (h,4k')=1}} \ee{-\frac{nh}{4k'}} \int_{\xi_{h,4k'}} \frac{F\left(\ee{2\tau'}\right)F\left(\ee{4\tau'}\right)}{F\left(\ee{\tau'}\right)^4} \nonumber\\
&\quad\quad\quad\quad\quad\quad\quad\quad\quad\quad\quad\quad\quad\quad\quad \times\omega_{h,k'}\ e^{\frac{\pi}{24k'}\left(\frac{1}{z}-z\right)}\ z^{-1} \ee{-n \phi} e^{2 \pi n \varrho}\ d\phi \nonumber\\
&= \sum_{1\le 4k'\le N}\sum_{\substack{0\le h\le 4k'\\ (h,4k')=1}} \ee{-\frac{nh}{4k'}} \int_{\xi_{h,4k'}} \omega_{h,k'}\ e^{\frac{\pi}{24k'}\left(\frac{1}{z}-z\right)}\ z^{-1} \ee{-n \phi} e^{2 \pi n \varrho}\ d\phi \nonumber\\
&\quad + \sum_{1\le 4k'\le N}\sum_{\substack{0\le h\le 4k'\\ (h,4k')=1}} \ee{-\frac{nh}{4k'}} \int_{\xi_{h,4k'}} \left(\frac{F\left(\ee{2\tau'}\right)F\left(\ee{4\tau'}\right)}{F\left(\ee{\tau'}\right)^4}-1\right) \nonumber\\
&\quad\quad\quad\quad\quad\quad\quad\quad\quad\quad\quad\quad\quad\quad\quad\quad \times\omega_{h,k'}\ e^{\frac{\pi}{24k'}\left(\frac{1}{z}-z\right)}\ z^{-1} \ee{-n \phi} e^{2 \pi n \varrho}\ d\phi \nonumber\\
&=:T_1+T_2.
\end{align*}

We estimate $T_2$:
\begin{align*}
|T_2|&\le \sum_{1\le 4k'\le N}\sum_{\substack{0\le h\le 4k'\\ (h,4k')=1}}  \int_{\xi_{h,4k'}} e^{2\pi\varrho \left(n-\frac{1}{12}\right)}\ \frac{N^2}{4k'}\ e^{\frac{\pi}{12}}\\
&\quad\quad\quad\quad\times \left(\exp\left(\frac{e^{-2\pi}}{\left(1-e^{-2\pi }\right)^2}+\frac{e^{-4\pi}}{\left(1-e^{-4\pi }\right)^2}+\frac{4e^{-\pi }}{\left(1-e^{-\pi }\right)^2}\right)-1\right)\ d\phi\\
&\le \sum_{1\le 4k'\le N}\sum_{\substack{0\le h\le 4k'\\ (h,4k')=1}}  e^{2\pi\varrho \left(n-\frac{1}{12}\right)}\ \frac{N^2}{4k'}\ e^{\frac{\pi}{12}}\\
&\quad\quad\quad\quad\times \left(\exp\left(\frac{e^{-2\pi}}{\left(1-e^{-2\pi }\right)^2}+\frac{e^{-4\pi}}{\left(1-e^{-4\pi }\right)^2}+\frac{4e^{-\pi }}{\left(1-e^{-\pi }\right)^2}\right)-1\right)\frac{2}{4k'N}\\
&\le \sum_{1\le 4k'\le N} e^{2\pi\varrho \left(n-\frac{1}{12}\right)}\ \frac{N^2}{4k'}\ e^{\frac{\pi}{12}}\\
&\quad\quad\quad\quad\times \left(\exp\left(\frac{e^{-2\pi}}{\left(1-e^{-2\pi }\right)^2}+\frac{e^{-4\pi}}{\left(1-e^{-4\pi }\right)^2}+\frac{4e^{-\pi }}{\left(1-e^{-\pi }\right)^2}\right)-1\right)\frac{2}{4k'N}\ 4k'\\
&\le 2\;e^{2\pi\varrho \left(n-\frac{1}{12}\right)} N\big(\log N+1\big) e^{\frac{\pi}{12}}\\
&\quad\quad\quad\quad\times \left(\exp\left(\frac{e^{-2\pi}}{\left(1-e^{-2\pi }\right)^2}+\frac{e^{-4\pi}}{\left(1-e^{-4\pi }\right)^2}+\frac{4e^{-\pi }}{\left(1-e^{-\pi }\right)^2}\right)-1\right)\\
&\le 0.6\;e^{2\pi\varrho \left(n-\frac{1}{12}\right)} N\big(\log N+1\big).
\end{align*}

Next, it follows from Lemma \ref{le:key-int} with $k=4k'$ that
\begin{align*}
T_1&=\sum_{1\le 4k'\le N}\sum_{\substack{0\le h\le 4k'\\ (h,4k')=1}} \ee{-\frac{nh}{4k'}} \int_{\xi_{h,4k'}} \omega_{h,k'}\ e^{\frac{\pi}{24k'}\left(\frac{1}{z}-z\right)}\ z^{-1} \ee{-n \phi} e^{2 \pi n \varrho}\ d\phi\\
&=D_n+\sum_{1\le 4k'\le N}\sum_{\substack{0\le h\le 4k'\\ (h,4k')=1}} \ee{-\frac{nh}{4k'}} \omega_{h,k'} \frac{2\pi}{4k'}\; I_0\left(\frac{2\pi \sqrt{n-\frac{1}{12}}}{4\sqrt{3}k'}\right),
\end{align*}
with
\begin{align*}
|D_n|\le \sum_{1\le 4k'\le N}\sum_{\substack{0\le h\le 4k'\\ (h,4k')=1}} 5.2\; \frac{e^{2\pi\varrho \left(n-\frac{1}{12}\right)}N}{n-\frac{1}{12}} \le \frac{5.2}{4}\; \frac{e^{2\pi\varrho \left(n-\frac{1}{12}\right)}N^3}{n-\frac{1}{12}}.
\end{align*}

For the main term in $T_1$, we compute the $k'=1$ case
\begin{align*}
\sum_{\substack{0\le h\le 4\\ (h,4)=1}} \ee{-\frac{nh}{4}} \omega_{h,1} \frac{2\pi}{4}&=\begin{cases}
-\pi & \text{if $n\equiv 1 \pmod{4}$},\\
\pi & \text{if $n\equiv 3 \pmod{4}$},\\
0 & \text{otherwise}
\end{cases}\\
&=:c_{1}(n),
\end{align*}
and the $k'=2$ case
\begin{align*}
\sum_{\substack{0\le h\le 8\\ (h,8)=1}} \ee{-\frac{nh}{8}} \omega_{h,2} \frac{2\pi}{8}&=\begin{cases}
\pi \sin\frac{\pi}{8} & \text{if $n\equiv 0 \pmod{8}$},\\
\pi \cos\frac{\pi}{8} & \text{if $n\equiv 2 \pmod{8}$},\\
-\pi \sin\frac{\pi}{8} & \text{if $n\equiv 4 \pmod{8}$},\\
-\pi \cos\frac{\pi}{8} & \text{if $n\equiv 6 \pmod{8}$},\\
0 & \text{otherwise}
\end{cases}\\
&=:c_{2}(n).
\end{align*}
Hence we obtain the main term in \eqref{eq:main-asy-b-1}
$$c_{1}(n) \; I_0\left(\frac{\pi \sqrt{n-\frac{1}{12}}}{2\sqrt{3}}\right)+c_{2}(n) \; I_0\left(\frac{\pi \sqrt{n-\frac{1}{12}}}{4\sqrt{3}}\right).$$

For $3\le k'\le N/4$, we use \eqref{add-I-upper} to deduce that
\begin{align*}
|H_n|:=&\left|\sum_{3\le k'\le N/4}\sum_{\substack{0\le h\le 4k'\\ (h,4k')=1}} \ee{-\frac{nh}{4k'}} \omega_{h,k'} \frac{2\pi}{4k'}\; I_0\left(\frac{2\pi \sqrt{n-\frac{1}{12}}}{4\sqrt{3}k'}\right)\right|\\
\le& \sum_{3\le k'\le N/4} 4k' \frac{2\pi}{4k'} \sqrt{\frac{\pi}{8}} \frac{e^{\frac{2\pi \sqrt{n-\frac{1}{12}}}{4\sqrt{3}k'}}}{\sqrt{\frac{2\pi \sqrt{n-\frac{1}{12}}}{4\sqrt{3}k'}}}\\
\le& \frac{3^{\frac{1}{4}}\pi}{8}N^{\frac{3}{2}} \left(n-\frac{1}{12}\right)^{-\frac{1}{4}}  e^{\frac{\pi \sqrt{n-\frac{1}{12}}}{6\sqrt{3}}}\\
\le& 0.6\;N^{\frac{3}{2}} \left(n-\frac{1}{12}\right)^{-\frac{1}{4}}  e^{\frac{\pi \sqrt{n-\frac{1}{12}}}{6\sqrt{3}}}.
\end{align*}

\subsection{The asymptotic formula of $b(n)$}

It follows that the total error term is
\begin{align*}
|E(n)|&\le|S_1|+|S_2|+|T_2|+|D_n|+|H_n|\\
& \le 32.9\; e^{2\pi\varrho \left(n-\frac{1}{12}\right)} N\big(\log N+1\big)+\frac{5.2}{4}\; \frac{e^{2\pi\varrho \left(n-\frac{1}{12}\right)}N^3}{n-\frac{1}{12}}\\
&\quad +0.6\;N^{\frac{3}{2}} \left(n-\frac{1}{12}\right)^{-\frac{1}{4}}  e^{\frac{\pi \sqrt{n-\frac{1}{12}}}{6\sqrt{3}}}.
\end{align*}
Setting $N=\sqrt{2\pi (n-1/12)}$ yields
\begin{align*}
|E(n)|& \le 32.9\; e \sqrt{2\pi} \sqrt{n-\frac{1}{12}}\; \left(\log \sqrt{2\pi \left(n-\frac{1}{12}\right)}+1\right)\\
&\quad+\frac{5.2\;e}{4} \left(\sqrt{2\pi}\right)^3\sqrt{n-\frac{1}{12}}\\
&\quad +0.6\;\left(\sqrt{2\pi}\right)^{\frac{3}{2}} \sqrt{n-\frac{1}{12}} \; e^{\frac{\pi \sqrt{n-\frac{1}{12}}}{6\sqrt{3}}}\\
&\le 224.2\; \sqrt{n-\frac{1}{12}}\; \left(\log \sqrt{2\pi \left(n-\frac{1}{12}\right)}+1\right) + 55.6\; \sqrt{n-\frac{1}{12}}\\
&\quad+ 2.4\; \sqrt{n-\frac{1}{12}} \; e^{\frac{\pi \sqrt{n-\frac{1}{12}}}{6\sqrt{3}}}.
\end{align*}

Together with the main term
$$c_{1}(n) \; I_0\left(\frac{\pi \sqrt{n-\frac{1}{12}}}{2\sqrt{3}}\right)+c_{2}(n) \; I_0\left(\frac{\pi \sqrt{n-\frac{1}{12}}}{4\sqrt{3}}\right),$$
we arrive at the desired asymptotic formula.

\subsection{Proof of Theorem \ref{THM:mod 4}}

We deduce from a short computation that the sign of $M^{*}(0,4,n)-M^{*}(2,4,n)$ is determined by the main term when $n\ge 2160$. For the remaining cases, we check them directly and hence arrive at the desired inequalities.

\section{Some elementary arguments}\label{sec:ele}

It is also possible to partially prove Theorems \ref{THM:mod 3} and \ref{THM:mod 4} with merely elementary arguments.

Let $f(q)=\sum\limits_{n=-\infty}^{\infty}a_{n}q^{n}$ and $g(q)=\sum\limits_{n=-\infty}^{\infty}b_{n}q^{n}$ be two formal power series in $q$. We write $f\succeq g$ (resp.~$f\succ g$) if $a_{n}\geq b_{n}$ (resp.~$a_{n}>b_{n}$) for all $n$.

\subsection{Two cases for $m=3$}

Here we give elementary proofs of the following two inequalities:
\begin{align}
M^{*}(0,3,3n)&>M^{*}(1,3,3n),\\
M^{*}(0,3,3n+1)&<M^{*}(1,3,3n+1).
\end{align}

Recall the following two identities:
\begin{lemma}
We have
\begin{align}
(q;q)_{\infty} &=(q^{12},q^{15},q^{27};q^{27})_{\infty}-q(q^{6},q^{21},q^{27};q^{27})_{\infty}-q^{2}(q^{3},q^{24},q^{27};q^{27})_{\infty},\label{3-dissection of f1}\\
(q;q)_{\infty}^{3} &=P(q^{3})-3q(q^{9};q^{9})_{\infty}^{3},\label{3-dissection of f3}
\end{align}
where
\begin{align}
P(q) &=(q;q)_{\infty}\left(1+6\sum_{n=0}^{\infty}\left(\dfrac{q^{3n+1}}{1-q^{3n+1}}-\dfrac{q^{3n+2}}{1-q^{3n+2}}\right)\right)\label{Hirshhorn iden}\\
 &=(q;q)_{\infty}\sum_{m,n=-\infty}^{\infty}q^{m^{2}+mn+n^{2}}.\label{cubic theta function}
\end{align}
\end{lemma}

Eq.~\eqref{3-dissection of f1} is a simple consequence of the Jacobi triple product identity \cite[p.~21, Theorem 2.8]{Andr1976}. Eq.~\eqref{3-dissection of f3} can be found in \cite{Wang2016}. Eq.~\eqref{Hirshhorn iden} comes from \cite{Hir2001}. At last, Eq.~\eqref{cubic theta function} comes from \cite{BBG1994}.

Applying \eqref{3-dissection of f1} and \eqref{3-dissection of f3} to \eqref{eq:gen-mod-3} yields
\begin{align}
 &\sum_{n=0}^{\infty}\left(M^{*}(0,3,n)-M^{*}(1,3,n)\right)q^{n}
 =\dfrac{(q;q)_{\infty}^{4}}{(q^{3};q^{3})_{\infty}^{2}}\nonumber\\
 &\quad =\dfrac{1}{(q^{3};q^{3})_{\infty}^{2}}\left(P(q^{3})-3q(q^{3};q^{3})_{\infty}^{3}\right)\nonumber\\
 &\quad\quad\times\big((q^{12},q^{15},q^{27};q^{27})_{\infty}-q(q^{6},q^{21},q^{27};q^{27})_{\infty}-q^{2}(q^{3},q^{24},q^{27};q^{27})_{\infty}\big).\label{gf:3-dissection}
\end{align}

Consequently, we dissect
\begin{align}
 &\sum_{n=0}^{\infty}\left(M^{*}(0,3,3n)-M^{*}(1,3,3n)\right)q^{n}\nonumber\\
 &\quad=\dfrac{1}{(q;q)_{\infty}^{2}}\big((q^{4},q^{5},q^{9};q^{9})_{\infty}P(q)+3q(q,q^{8},q^{9};q^{9})_{\infty}(q^{3};q^{3})_{\infty}^{3}\big).\label{gf:3n}
\end{align}
Note that
\begin{align*}
\dfrac{(q^{4},q^{5},q^{9};q^{9})_{\infty}P(q)}{(q;q)_{\infty}^{2}} &=\dfrac{1}{(q,q^{2},q^{3},q^{6},q^{7},q^{8};q^{9})_{\infty}}\sum_{m,n=-\infty}^{\infty}q^{m^{2}+mn+n^{2}}\\
 &\succeq\dfrac{1}{1-q}\sum_{m,n=-\infty}^{\infty}q^{m^{2}+mn+n^{2}}\succeq\dfrac{1}{1-q}=\sum_{k=0}^{\infty}q^{k}\succ0,
\end{align*}
and that
\begin{align*}
 \dfrac{(q,q^{8},q^{9};q^{9})_{\infty}(q^{3};q^{3})_{\infty}^{3}}{(q;q)_{\infty}^{2}} &=\dfrac{1}{(q^{2},q^{3},q^{4},q^{5},q^{6},q^{7};q^{9})_{\infty}}\cdot\dfrac
 {(q^{3};q^{3})_{\infty}^{3}}{(q;q)_{\infty}}\succeq0,
\end{align*}
since $\frac{(q^{3};q^{3})_{\infty}^{3}}{(q;q)_{\infty}}$ is the generating function of $3$-core partitions. Hence from \eqref{gf:3n} we deduce that $M^{*}(0,3,3n)>M^{*}(1,3,3n)$ for all $n\geq0$.

We also dissect
\begin{align*}
 &\sum_{n=0}^{\infty}\left(M^{*}(0,3,3n+1)-M^{*}(1,3,3n+1)\right)q^{n}\\
 &\quad=-\dfrac{3(q^{4},q^{5},q^{9};q^{9})_{\infty}(q^{3};q^{3})_{\infty}^{3}+(q^{2},q^{7},q^{9};q^{9})_{\infty}P(q)}{(q;q)_{\infty}^{2}}.
\end{align*}
Using similar arguments, one may prove that $M^{*}(0,3,3n+1)<M^{*}(1,3,3n+1)$ for all $n\geq0$.

\subsection{Two cases for $m=4$}

We also give elementary proofs of the following two inequalities:
\begin{align}
M^{*}(0,4,4n+1)&<M^{*}(2,4,4n+1),\label{eq:ineq-mod4-4n-1}\\
M^{*}(0,4,4n+3)&>M^{*}(2,4,4n+3).\label{eq:ineq-mod4-4n-3}
\end{align}

We know from \cite[p.~40, Entry 25]{Ber1991} that
\begin{lemma}
It holds that
\begin{align}
(q;q)_{\infty}^{4}=\dfrac{(q^{4};q^{4})_{\infty}^{10}}{(q^{2};q^{2})_{\infty}^{2}(q^{8};q^{8})_{\infty}^{4}}-4q\dfrac{(q^{2};q^{2})_{\infty}^{2}(q^{8};q^{8})_{\infty}^{4}}
{(q^{4};q^{4})_{\infty}^{2}}.\label{f1^4}
\end{align}
\end{lemma}

Applying \eqref{f1^4} to \eqref{eq:gen-mod-4} gives
\begin{align}\label{add-M4}
\sum_{n=0}^{\infty}\left(M^{*}(0,4,n)-M^{*}(2,4,n)\right)q^{n}=\dfrac{(q^{4};q^{4})_{\infty}^{9}}{(q^{2};q^{2})_{\infty}^{3}(q^{8};q^{8})_{\infty}^{4}}-4q\dfrac
{(q^{2};q^{2})_{\infty}(q^{8};q^{8})_{\infty}^{4}}{(q^{4};q^{4})_{\infty}^{3}}.
\end{align}
Hence,
\begin{align*}
\sum_{n=0}^{\infty}\left(M^{*}(0,4,2n+1)-M^{*}(2,4,2n+1)\right)q^{n}=-4\dfrac{(q;q)_{\infty}(q^{4};q^{4})_{\infty}^{4}}{(q^{2};q^{2})_{\infty}^{3}}:=g(q).
\end{align*}
To prove \eqref{eq:ineq-mod4-4n-1} and \eqref{eq:ineq-mod4-4n-3}, it suffices to show that the coefficient of $q^{n}$ in $g(-q)$ is negative for all $n\ge 0$. Note that
\begin{align*}
g(-q) &=-4\dfrac{(-q;-q)_{\infty}(q^{4};q^{4})_{\infty}^{4}}{(q^{2};q^{2})_{\infty}^{3}}=-4\dfrac{(q^{4};q^{4})_{\infty}^{4}}{(q;q)_{\infty}(q^{2};q^{2})_{\infty}}\\
 &=-4\dfrac{(q^{2};q^{2})_{\infty}}{(q;q)_{\infty}}\left(\dfrac{(q^{4};q^{4})_{\infty}^{2}}{(q^{2};q^{2})_{\infty}}\right)^{2}=-4\dfrac{1}{(q;q^{2})_{\infty}}
 \left(\sum_{m=0}^{\infty}q^{m(m+1)}\right)^{2}\prec0,
\end{align*}
where the last equality follows from the Jacobi triple product identity. We hence arrive at \eqref{eq:ineq-mod4-4n-1} and \eqref{eq:ineq-mod4-4n-3}.

\begin{remark}
As an immediate consequence, we obtain the following congruence.
\begin{corollary}
For all $n\geq0$,
\begin{align*}
M^{*}(0,4,2n+1)\equiv M^{*}(2,4,2n+1)\pmod{4}.
\end{align*}
\end{corollary}
\end{remark}

\section{Final remarks}
We conclude with two questions that will merit further investigation.
\begin{enumerate}[1.]
\item In view of \eqref{gf:3-dissection}, one can readily obtain
\begin{align}
&\sum_{n=0}^{\infty}\left(M^{*}(0,3,3n+2)-M^{*}(1,3,3n+2)\right)q^n\nonumber\\
&\qquad=\frac{(q^9;q^9)_{\infty}}{(q;q)_{\infty}^2}\left(3(q^2,q^7;q^9)_{\infty}(q^3;q^3)_{\infty}^3
-(q,q^8;q^9)_{\infty}P(q) \right). \label{add-M01-3}
\end{align}
However, it is unclear whether one can find an elementary proof for the case $n\equiv 2 \pmod{3}$ of \eqref{mod 3:pos} by analyzing \eqref{add-M01-3}.

Moreover, from \eqref{add-M4} we have
\begin{align}\label{add-M04-2n}
\sum_{n=0}^{\infty}\left(M^{*}(0,4,2n)-M^{*}(2,4,2n)\right)q^{n}=\frac{(q^2;q^2)_{\infty}^9}{(q;q)_{\infty}^3(q^4;q^4)_{\infty}^4}.
\end{align}
From \cite[Lemma 4.1]{ABCKM} we find
\begin{align}\label{add-2-dis}
\frac{1}{(q;q)_{\infty}}=\frac{(q^{16};q^{16})_{\infty}}{(q^2;q^2)_{\infty}^2}\left((-q^6,-q^{10};q^{16})_{\infty}+q(-q^2,-q^{14};q^{16})_{\infty} \right).
\end{align}
Substituting \eqref{add-2-dis} into \eqref{add-M04-2n}, extracting the terms in which the exponent of $q$ is even and odd, respectively, after replacing $q^2$ by $q$ and then changing $q$ to $-q$, we obtain
\begin{align}
&\sum_{n=0}^{\infty}\left(M^{*}(0,4,4n)-M^{*}(2,4,4n)\right)(-q)^{n}\nonumber\\
&=\dfrac{(q^{2};q^{2})_{\infty}^{2}(q^{8};q^{8})_{\infty}}{(q;q)_{\infty}^{2}(q^{4};q^{4})_{\infty}^{2}}\left((q^{3},q^{5};q^{8})_{\infty}\dfrac{(q^{4};q^{4})_{\infty}
 ^{5}}{(q^{8};q^{8})_{\infty}^{2}}-2q(q,q^{7};q^{8})_{\infty}\dfrac{(q^{2};q^{2})_{\infty}^{2}(q^{8};q^{8})_{\infty}^{2}}{(q^{4};q^{4})_{\infty}}\right),\label{gf:4n}\\
& \sum_{n=0}^{\infty}\left(M^{*}(0,4,4n+2)-M^{*}(2,4,4n+2)\right)(-q)^{n}\nonumber\\
&=\dfrac{(q^{2};q^{2})_{\infty}^{2}(q^{8};q^{8})_{\infty}}{(q;q)_{\infty}^{2}(q^{4};q^{4})_{\infty}^{2}}\left(2(q^{3},q^{5};q^{8})_{\infty}\dfrac{(q^{2};q^{2})_{\infty}
 ^{2}(q^{8};q^{8})_{\infty}^{2}}{(q^{4};q^{4})_{\infty}}+(q,q^{7};q^{8})_{\infty}\dfrac{(q^{4};q^{4})_{\infty}
 ^{5}}{(q^{8};q^{8})_{\infty}^{2}}\right).\label{gf:4n+2}
\end{align}
It would be appealing to study the sign patterns of the coefficients on the left-hand sides of \eqref{gf:4n} and \eqref{gf:4n+2}.

\item According to the second author's recent work with Fu \cite{FT2018-2} involving multiranks, we notice that for the purpose of combinatorially interpreting
congruences \eqref{partition pairs mod 5}, there are a number of alternatives that will work equally well as Garvan's bicrank. Namely,
\begin{align*}
\textrm{bicrank}_{\ell}:=c^{*}(\pi_{1})+\ell c^{*}(\pi_{2})\quad\quad \textrm{for~}\ell\equiv2,3\pmod{5},
\end{align*}
where we refer to \cite[Sect.~6]{Gar2010} for the definition of $c^{*}(\pi)$. We remark that the $\ell=1$ and $2$ cases respectively correspond to Andrews' and Garvan's bicranks. Following the same line of proving Theorems \ref{th:main-asy} and \ref{th:main-asy-b}, one may obtain a number of asymptotic formulas and inequalities for these partition statistics involving $2$-colored partitions.
\end{enumerate}

\section*{Acknowledgement}
The second author was supported by the National Natural Science Foundation of China grant 11501061.  The third author was partially supported by ``the Fundamental Research Funds for the Central Universities'' and a start-up research grant of the Wuhan University.

\bibliographystyle{amsplain}

\end{document}